\documentclass[aos,nameyear,dvips]{arximspdf}
\usepackage{dcolumn}
\usepackage{graphicx}

\doi{10.1214/10-AOS804}
\volume{39}
\issue{1}
\pubyear{2011}
\firstpage{418}
\lastpage{438}

\makeatletter

\newcolumntype{d}[1]{D{.}{.}{#1}}

\newcommand{\dd}{{d}}
\newcommand{\pr}{\mathbb{P}}
\newcommand{\esp}{\mathbb{E}}

\newtheorem{theorem}{Theorem}
\newproclaim{definition}{Definition}
\newtheorem{lem}[theorem]{Lemma}
\newtheorem{corollary}[theorem]{Corollary}
\newproclaim{assumption}{Assumption}
\newproclaim{remark}{Remark}
\newproclaim{example}{Example}

\makeatother

\begin{document}
\begin{frontmatter}

\title{Monotone spectral density estimation}
\runtitle{Monotone spectral density estimation}

\begin{aug}
\author[A]{\fnms{Dragi} \snm{Anevski}\corref{}\ead[label=e1]{dragi@maths.lth.se}} and
\author[B]{\fnms{Philippe} \snm{Soulier}\ead[label=e2]{philippe.soulier@u-paris10.fr}}
\runauthor{D. Anevski and P. Soulier}
\affiliation{Lund University and Universit\'e Paris Ouest Nanterre}
\address[A]{Centre for Mathematical Sciences\\
Box 118\\
SE-22100 Lund\\
Sweden\\
\printead{e1}} %adresu isvedimo komanda gale!
\address[B]{Department of Mathematics\\
University Paris X\\
B\^{a}timent G, Bureau E18\\
200 avenue de la R\'{e}publique\\
92000 Nanterre Cedex\\
France\\
\printead{e2}}
\end{aug}

% HISTORY:
\received{\smonth{1} \syear{2009}}
\revised{\smonth{1} \syear{2010}}

% ABSTRACT
%
\begin{abstract}
We propose two estimators of a monotone spectral
density, that are based on the periodogram. These are the isotonic
regression of the periodogram and the isotonic regression of the
log-periodogram. We derive pointwise limit distribution results for
the proposed estimators for short memory linear processes and long
memory Gaussian processes and also that the estimators are rate
optimal.
\end{abstract}

% KEYWORDS
%

\end{frontmatter}

%s1 ###
\section{Introduction}
The motivation for doing spectral analysis of stationary time series
comes from the need to analyze the frequency content in the signal.
The frequency content can for instance be described by the spectral
density, defined below, for the process. One could be interested in
looking for a few dominant frequencies or frequency regions, which
correspond to multimodality in the spectral density. Inference methods
for multimodal spectral densities have been treated in
\citet{davieskovac2004}, using the taut string method. A simpler
problem is that of fitting a unimodal spectral density, that is, the
situation when there is only one dominant frequency, which can be
known or unknown, corresponding to known or unknown mode,
respectively, and leading to the problem of fitting a unimodal
spectral density to the data. In this paper we treat unimodal spectral
density estimation for known mode. A spectral density that is
decreasing on $[0,\pi]$ is a model for the frequency content in the
signal being ordered. A unimodal spectral density is a model for there
being one major frequency component, with a decreasing amount of other
frequency components seen as a function of the distance to the major
frequency.

Imposing monotonicity (or unimodality) means that one imposes a
nonparametric approach, since the set of monotone (or unimodal)
spectral densities is infinite dimensional. A parametric problem that
is contained in our estimation problem is that of a power law
spectrum, that is, when one assumes that the spectral density decreases as
a power function $f(u)\sim u^{-\beta}$ for $u\in(0,\pi)$, with unknown
exponent~$\beta$. Power law spectra seem to have important
applications to physics, astronomy and medicine; four different
applications mentioned in \citet{mccoywaldenpercival1998} are: (a)
fluctuations in the Earth's rate of rotation [cf.
\citet{munkmacdonald1975}], (b) voltage fluctuations across cell
membrane [cf. \citet{holden1976}], (c) time series of impedances
of rock
layers in boreholes [cf., e.g., \citet{kernerharris1994}] and
(d) x-ray
time variability of galaxies [cf. \citet{mchardyczerny1987}]. We
propose to use a nonparametric approach as an alternative to the power
law spectrum methods used in these applications. There are (at least)
two reasons why this could make sense: first, the reason for using a
power function, for example, to model the spectrum in the background radiation,
is (at best) a theoretical consideration exploiting physical theory
and leading to the power function as a good approximation. However,
this is a stronger model assumption to impose on the data than merely
imposing monotonicity, and thus one could imagine a wider range of
situations that should be possible to analyze using our methods.
Second, fitting a power law spectral model to data consists of doing
linear regression of the log periodogram; if the data are not very
well aligned along a straight line (after a log-transformation) this
could influence the overall fit. A nonparametric approach, in which
one assumes only monotonicity, is more robust against possible misfit.

Sometimes one assumes a piecewise power law spectrum [cf.
\citet{percival1991}] as a model. Our methods are well adapted to
these situations when the overall function behavior is that of a
decreasing function.

Furthermore there seem to be instances in the literature when a
monotonically decreasing (or monotonically increasing) spectral
density is both implicitly assumed as a model, and furthermore seems
feasible: two examples in \citet{percivalwalden1993} [cf., e.g.,
Figures 20 and 21 in \citet{percivalwalden1993}] are (e) the wind
speed in a certain direction at a certain location measured every
0.025 second (for which a decreasing spectral density seems to be
feasible) and (f) the daily record of how well an atomic clock keeps
time on a day-to-day basis (which seems to exhibit an increasing
spectral density). The methods utilized in \citet{percivalwalden1993}
are smoothing of the periodogram. We propose to use an
order-restricted estimator of the spectral density, and would like to
claim that this is better adapted to the situations at hand.

Decreasing spectral densities can arise when one observes a sum of
several parametric time series, for instance, AR(1) processes with
coefficient $|a|<1$; the interest of the nonparametric method in that
case is that one does not have to know how many AR(1) are summed up.
Another parametric example is an $\operatorname{ARFIMA}(0,d,0)$ with
$0<d<1/2$, which
has a decreasing spectral density, which is observed with added white
noise, or even with added one (or several) AR(1) processes; the
resulting time series will have a decreasing spectral density. Our
methods are well adapted to this situation, and we will illustrate the
nonparametric methods on simulated data from such parametric models.

The spectral measure of a weakly stationary process is the positive
measure $\sigma$ on $[-\pi,\pi]$ characterized by the relation
\[
\operatorname{cov}(X_0,X_k) = \int_{-\pi}^\pi e^{ i k x}
\sigma(\dd x) .
\]
The spectral density, when it exists, is the density of $\sigma$ with
respect to Lebesgue's measure. It is an even nonnegative integrable
function on $[-\pi,\pi]$. Define the spectral distribution function on
$[-\pi,\pi]$ by
\begin{eqnarray*}
F(\lambda)&=&\int_{0}^{\lambda} f(u) \,\dd u ,\qquad 0\leq\lambda\leq
\pi,\\
F(\lambda)&=&-F(-\lambda) ,\qquad -\pi\leq\lambda< 0 .
\end{eqnarray*}
An estimate of the spectral density is given by the periodogram
\[
I_n(\lambda) = \frac{1}{2\pi n} \Biggl| \sum_{k=1}^{n} X_k
{e}^{-{i} k \lambda} \Biggr|^2 .
\]
The spectral distribution function is estimated by the empirical
spectral distribution function
\[
F_n(\lambda) = \int_0^{\lambda} I_n(u) \,{d} u .
\]
Functional central limit theorems for $F_n$ have been established in
\citet{dahlhaus1989} and \citet{mikoschnorvaisa1997}.
However, since
the derivative is not a smooth map, the properties of $F_n$ do not
transfer to $I_n$, and furthermore it is well known that the
periodogram is not even a consistent estimate of the spectral density.
The standard remedy for obtaining consistency is to use kernel
smoothers. This, however, entails a bandwidth choice, which is somewhat
ad hoc. The assumption of monotonicity allows for the construction of
adaptive estimators that do not need a pre-specified bandwidth.

We will restrict our attention to the class of nonincreasing functions.
\begin{definition}
Let $\mathcal{F}$ be the convex cone of integrable, monotone
nonincreasing functions on $(0,\pi]$.
\end{definition}

Given a stationary sequence $\{X_k\}$ with spectral density $f$, the
goal is to estimate $f$ under the assumption that it lies in $\mathcal
{ F}$. We suggest two estimators, which are the ${\mathbb L}^2$
orthogonal projections on the convex cone $\mathcal{F}$ of the
periodogram and of the log-periodogram, respectively.
\begin{longlist}
\item The $\mathbb L^2$ minimum distance estimate between the
periodogram and $\mathcal{F}$ is defined as
%
%e1 ###
\begin{equation} \label{eq:def-hatfn}
\hat{f}_n = \mathop{\arg\min}_{z\in\mathcal{F}} Q(z) ,
\end{equation}
with
\[
Q(z) = \int_0^{\pi}\bigl(I_n(s)-z(s)\bigr)^2 \,\dd s .
\]
This estimator of the spectral density naturally yields a
corresponding estimator $\hat F_n$ of the spectral distribution
function $F$, defined by
%
%e2 ###
\begin{equation} \label{eq:def-hatFn}
\hat F_n(t) = \int_0^t \hat f_n(s) \,\dd s .
\end{equation}
\item The $\mathbb L^2$ minimum distance estimate between the
log-periodogram (often called the cepstrum) and the ``logarithm of
$\mathcal{F}$'' is defined as
%
%e3 ###
\begin{equation} \label{eq:def-tildefn}
\tilde{f}_n = \exp\mathop{\arg\min}_{z\in\mathcal{F}}\tilde{Q}(z) ,
\end{equation}
with
\[
\tilde{Q}(z) = \int_0^{\pi} \{\log I_n(s)+\gamma-\log z(s)\}^2
\,\dd s ,
\]
where $\gamma$ is Euler's constant. To understand the occurrence of
the centering $-\gamma$, recall that if $\{X_n\}$ is a Gaussian white
noise sequence with variance $\sigma^2$, then its spectral density is
$\sigma^2/(2\pi)$ and the distribution of $I_n(s)/(\sigma^2/2\pi)$ is
a standard exponential (i.e., one half of a chi-square with two degrees
of freedom), and it is well known that if $Z$ is a standard
exponential, then $\esp[\log(Z)]=-\gamma$ and $\operatorname{var}(\log Z) =
\pi^2/6$ [see, e.g., \citet{hurvichdeobrodsky1998}]. The log-spectral
density is of particular interest in the context of long-range
dependent time series, that is, when the spectral density has a
singularity at some frequency and might not be square integrable,
though it is always integrable by definition. For instance, the
spectral density of an $\operatorname{ARFIMA}(0,d,0)$ process is
$f(x)=\sigma^2|1-e^{i x}|^{-2d}$, with
$d\in(-1/2,1/2)$. It is decreasing on $(0,\pi]$ for $d\in(0,1/2)$ and
not square integrable for $d\in(1/4,1/2)$. In this context, for
technical reasons, we will take $I_n$ to be a step function changing
value at the so-called Fourier frequencies $\lambda_k = 2\pi k/n$.
\end{longlist}

The paper is organized as follows: in Section \ref{sec:derivation} we
derive the algorithms for the estimators $\hat{f}_n$, $\hat F_n$ and
$\tilde{f}_n$. In Section \ref{sec:lowerbound} we derive a lower bound
for the asymptotic local minimax risk in monotone spectral density
estimation and show that the rate is not faster than $n^{-1/3}$. In
Section \ref{sec:limit-distr} we derive the pointwise limit
distributions for the proposed estimators. The limit distribution of
$\hat{f}_n$ (suitably centered and normalized) is derived for a linear
process. The asymptotic distribution is that of the slope of the
least concave majorant at 0 of a quadratic function plus a two-sided
Brownian motion. Up to constants, this distribution is the so-called
Chernoff's distribution [see \citet{groeneboomwellner2001}] which
turns up in many situations in monotone function estimation [see, e.g.,
\citet{prakasarao1969} for monotone density estimation and
\citet{wright1981} for monotone regression function estimation]. The
limit distribution for $\tilde{f}_n$ is derived for a Gaussian
process and is similar to the result for~$\hat{f}_n$.
Section \ref{sec:simulations} contains a simulation study with
plots of the estimators. Section~\ref{section:proofs} contains
the proofs of the limit distribution results (Theorems
\ref{theo:lineaire-monotone} and \ref{theo:gaussien-lrd-monotone}).

%s2 ###
\section{Identification of the estimators}
\label{sec:derivation}

Let $h$ be a function defined on a compact interval $[a,b]$. The least
concave majorant $T(h)$ of $h$ and its derivative $T(h)'$ are defined
by
\begin{eqnarray*}
T(h) &=& \arg\min\{z\dvtx z \geq x, z \mbox{ concave} \} , \\
T(h)'(t) &=& \min_{u < t} \max_{v\geq t} \frac{h(v) - h(u)}{v-u} .
\end{eqnarray*}
By definition, $T(h)(t) \geq h(t)$ for all $t\in[a,b]$, and it is also
clear that $T(h)(a) = h(a)$, $T(h)(b) = h(b)$. Since $T(h)$ is
concave, it is everywhere left and right differentiable, $T(h)'$ as
defined above coincides with the left derivative of $T(h)$ and
$T(h)(t) = \int_a^t T(h)'(s) \,\dd s$ [see, e.g.,
\citet{hormander2007}, Theorem 1.1.9]. We will also need the
following result.
\begin{lem} \label{lem:charac-concmaj}
If $h$ is continuous, then the support of the Stieltjes measure $\dd
T(h)'$ is included in the set $\{T(h)=h\}$.
\end{lem}
\begin{pf}
Since $h$ and $T(h)$ are continuous and $T(h)(a)-h(a) = T(h)(b)-h(b)
= 0$, the set $\{T(h)>h\}$ is open. Thus it is a union of open
intervals. On such an interval, $T(h)$ is linear since otherwise it
would be possible to build a concave majorant of $h$ that would be
strictly smaller than $T(h)$ on some smaller open subinterval. Hence
$T(h)'$ is piecewise constant on the open set $\{T(h)>h\}$, so that
the support of $\dd T(h)'$ is included in the closed set
$\{T(h)=h\}$.
\end{pf}

The next lemma characterizes the least concave majorant as the
solution of a quadratic optimization problem. For any integrable
function $g$, define the function $\bar g$ on $[0,\pi]$ by
\[
\bar g(t) = \int_0^t g(s) \,\dd s .
\]
\begin{lem} \label{lem:char-quadrat}
Let $g \in\mathbb L^2([0,\pi])$. Let $G$ be defined on
$\mathbb{L}^2([0,\pi])$ by
\[
G(f) = \|f-g\|_2^2 = \int_0^\pi\{f(s) - g(s)\}^2 \,\dd s .
\]
Then $\arg\min_{f\in\mathcal{F}} G(f) = T(\bar g)'$.
\end{lem}

This result seems to be well known. It is cited, for example, in Mammen
[(\citeyear{mammen1991}), page 726] but since we have not found a
proof, we give one for
completeness.

Let $G\dvtx\mathcal{F}\mapsto{\mathbb R}$ be an arbitrary functional. It is
called Gateaux differentiable at the point $f\in\mathcal F$ if the
limit
\[
G'_{f}(h) = \lim_{t\rightarrow0} \frac{G(f+th)-G(f)}{t}
\]
exists for every $h$ such that $f+th \in\mathcal F$ for small enough
$t$.

\begin{pf*}{Proof of Lemma \ref{lem:char-quadrat}}
Denote $G(f) = \|f-g\|^2_2$ and $\hat f = T(\bar g)'$. The Gateaux
derivative of $G$ at $\hat f$ in the direction $h$ is
\[
G'_{\hat f}(h) = 2 \int_0^\pi h(t) \{\hat f(t) - g(t)\} \,\dd t .
\]
By integration by parts, and using that $T(\bar g)(\pi)-\bar
g(\pi)=T(\bar g)(0)-\bar g(0) = 0$, for any function of bounded
variation $h$, we have
%
%e4 ###
\begin{equation} \label{eq:bv}
G'_{\hat f}(h) = - 2 \int_0^\pi\{T(\bar g)(t) - \bar g(t)\} \,\dd h(t) .
\end{equation}
By Lemma \ref{lem:charac-concmaj}, the support of the measure $\dd
\hat f$ is included in the closed set $\{T(\bar g)=\bar g\}$, and thus
%
%e5 ###
\begin{equation} \label{eq:nul}
G'_{\hat f}(\hat f) = - 2 \int_0^\pi\{T(\bar g)(t) - \bar
g(t)\} \,\dd\hat f(t) = 0 .
\end{equation}
If $h=f-\hat f$, with $f$ monotone nonincreasing, (\ref{eq:bv})
and (\ref{eq:nul}) imply that
%
%e6 ###
\begin{equation} \label{eq:positif}
G_{\hat f}'(f-\hat f) = - 2 \int_0^\pi\{T(\bar g)(t) - \bar g(t)\}
\,\dd f(t) \geq0 .
\end{equation}
Let $f\in\mathcal F$\vspace*{2pt} be arbitrary, and let $u$ be the function defined
on $[0,1]$ by $u(t) = G(\hat f+t(f-\hat f))$. Then $u$\vspace*{-1pt} is convex and
$u'(0) = G'_{\hat f}(f-\hat f) \geq0$ by (\ref{eq:positif}). Since
$u$ is convex, if $u'(0)\geq0$, then $u(1)\geq u(0)$, that is,
$G(f)\geq
G(\hat f)$. This proves that $\hat f = \arg\min_{f\in\mathcal{F}}
G(f)$.
\end{pf*}

Since $\hat f_n$ and $\log\tilde f_n$ are the minimizers of the
$\mathbb L^2$ distance of $I_n$ and $\log(I_n)+\gamma$,
respectively, over the convex cone of monotone functions, we can apply
Lemma \ref{lem:char-quadrat} to derive characterizations of
$\hat{f}_n$ and $\tilde{f}_n$.
\begin{theorem}
Let $\hat f_n$, $\hat F_n$ and $\tilde f_n$ be defined
in (\ref{eq:def-hatfn}), (\ref{eq:def-hatFn})
and (\ref{eq:def-tildefn}), respectively. Then
\begin{eqnarray*}
\hat{f}_{n} & = & T(F_{n})' , \\
\hat F_n(t) & = & T(F_n) , \\
\tilde{f}_{n} & = & \exp\{T(\tilde{F}_{n})'\} ,
\end{eqnarray*}
where
\begin{eqnarray*}
F_{n}(t)&=&\int_{0}^{t} I_{n}(u) \,\dd u ,\\
\tilde{F}_{n}(t) & = & \int_{0}^{t} \{\log I_{n}(u) +\gamma\} \,\dd u .
\end{eqnarray*}
\end{theorem}

Standard and well-known algorithms for calculating the map $y\mapsto
T(y)'$ are the pool adjacent violators algorithm (PAVA), the minimum
lower set algorithm (MLSA) and the min--max formulas; cf.
\citet{RobWriDyk88}. Since the maps $T$ and $T'$ are continuous
operations, in fact the algorithms PAVA and MLSA will be
approximations that solve the discrete versions of our problems,
replacing the integrals in $Q$ and $\tilde{Q}$ with approximating
Riemann sums. Note that the resulting estimators are order-restricted
means; the discrete approximations entail that these are approximated
as sums instead of integrals. The approximation errors are similar to
the ones obtained, for example, for the methods in
\citet{mammen1991} and \citet{anevskihossjer2006}.

%s3 ###
\section{Lower bound for the local asymptotic minimax risk}
\label{sec:lowerbound}
We establish a lower bound for the minimax risk when estimating a
monotone spectral density at a fixed point. This result will be
proved by looking at parametrized subfamilies of spectral densities in
an open set of densities on ${\mathbb R}^n$; the subfamilies can be
seen as (parametrized) curves in the set of monotone spectral
densities. The topology used will be the one generated by the metric
\[
{\rho}(f,g) = \int_{{\mathbb R}}^{}|f(x)-g(x)|\,dx
\]
for $f,g$ spectral density functions on $[-\pi,\pi] $. Note first that
the distribution of a stochastic process is not uniquely defined by
the spectral density. To accomodate this, let $\mathcal{L}_g$ be the set
of all laws of stationary processes (i.e., the translation invariant
probability distributions on ${\mathbb R}^{\infty}$) with spectral
density~$g$.

Let ${\varepsilon}>0,c_{1},c_{2}$ be given finite constants, and let
$t_{0}>0$, the point at which we want to estimate the spectral
density, be given.
\begin{definition}
For each $n\in{\mathbb Z}$ let $\mathcal{G}^{1}:=\mathcal{
G}^{1}({\varepsilon},c_{1},c_{2},t_{0})$ be a set of monotone $C^{1}$
spectral densities $g$ on $[0,\pi]$, such that
%
%e8 ###
%e7 ###
\begin{eqnarray}
\label{eq:A}
&\displaystyle\sup_{|t-t_0|<\varepsilon} g'(t)  <  0 ,& \\
\label{eq:B}
&\displaystyle c_{1} < \inf_{|t-t_{0}|<{\varepsilon}} g(t)  <
\sup_{|t-t_{0}|<{\varepsilon}} g(t) < c_{2} .&
\end{eqnarray}
\end{definition}
\begin{theorem} \label{theo:lowerbound}
For every open set $U$ in $\mathcal{G}^{1}$ there is a positive
constant $c(U)$ such that
\[
\liminf_{n\to\infty} \inf_{T_{n}} \sup_{g\in U} \sup_{L \in
\mathcal{ L}_g} n^{2/3} \esp_{L} \bigl[\bigl(T_{n}-g(t_{0})\bigr)^{2}\bigr] \geq c(U) ,
\]
where the infimum is taken over all functions $T_n$ of the data.
\end{theorem}
\begin{pf}
Let $k$ be a fixed real valued continuously differentiable function,
with support $[-1,1]$ such that $\int k(t) \,dt=0, k(0)=1$ and ${\sup}
|k(t)|\leq1$. Then, since $k'$ is continuous with compact support,
$|k'|<C$ for some constant $C<\infty$.

For fixed $h>0$, define a parametrized family of spectral densities
$g_{{\theta}}$ by
\[
g_{{\theta}}(t) = g(t) + {\theta} k \biggl(\frac{t-t_{0}}{h} \biggr) .
\]
Obviously, $\{g_{{\theta}}\}_{\theta\in\Theta}$ are $C^{1}$ functions.
Since
\[
g_{{\theta}}'(t)=g'(t)+\frac{{\theta}}{h}k' \biggl(\frac{t-t_{0}}{h} \biggr) ,
\]
and since $k'$ is bounded, we have that, for $|t-t_0|<\varepsilon$,
$g_{{\theta}}'(t)<0$ if $|{{\theta}}/{h}|<{\delta}$, for some
${\delta}=\delta(C)>0$. Thus, in order to make the parametrized
spectral densities $g_{\theta}$ strictly decreasing in the
neighborhood $\{t\dvtx|t-t_0|<\varepsilon\}$, the parameter space for
${\theta}$ should be chosen as
\[
{\Theta} =(-{\delta}h,{\delta}h).
\]
We will use the van Trees inequality [cf.
\citet{gilllevit1995}, Theorem 1]
for the estimand $g_{{\theta}}(t_{0})= g(t_{0})+{\theta}$. Let
${\lambda}$ be an arbitrary prior density on ${\Theta}$. Then, for
sufficiently small ${\delta}$, $\{g_{\theta}\dvtx\theta\in
\Theta\}\subset U$ (cf. the definition of the metric $\rho$). Let
$P_{\theta}$ denote the distribution of a Gaussian process with spectral
density $g_{\theta}$, and $\esp_\theta$ the corresponding
expectation. Then
\begin{eqnarray*}
\sup_{g \in U}\sup_{L\in\mathcal{L}_g}{\mathbb
E}_{L}\bigl(T_{n}-g(t_{0})\bigr)^{2} & \geq & \sup_{{\theta}\in{\Theta}} {
\mathbb E}_{{{\theta}}}\bigl(T_{n}-g_{{\theta}}(t_{0})\bigr)^{2} \\
& \geq & \int_{{\Theta}} E_{{\theta}}
\bigl(T_{n}-g_{\theta}(t_{0})\bigr)^{2}{\lambda}({\theta}) \,d{\theta} .
\end{eqnarray*}
Then, by the Van Trees inequality, we obtain
%
%e9 ###
\begin{equation} \label{eq:vt}
\int_{{\Theta}} \mathbb E_{{\theta}}
\bigl(T_{n}-g_{\theta}(t_{0})\bigr)^{2}{\lambda}({\theta}) \,d{\theta}
\geq\frac{1} {\int I_n(\theta)\lambda(\theta) \,\dd\theta+
\tilde{I}(\lambda) },
\end{equation}
where
\[
I_n(\theta) = \tfrac12 \operatorname{tr} ( \{M_n^{-1}(g_\theta)
M_n(\partial_\theta g_\theta) \}^2 )
\]
is the Fisher information matrix [cf. \citet{dzhaparidze1986}] with
respect to the parameter $\theta$ of a Gaussian process with spectral
density $g_\theta$, and for any even nonnegative integrable function
$\phi$ on $[-\pi,\pi]$, $M_n(\phi)$ is the Toeplitz matrix of order~$n$
\[
M_n(\phi)_{i,j} = \int_{-\pi}^\pi\phi(x) \cos\bigl((i-j)x\bigr) \,\dd x .
\]
For any $n\times n$ nonnegative symmetric matrix $A$, define
the spectral radius of $A$ as
\[
\rho(A) = \sup\{ u^t A u \mid u^tu=1\} ,
\]
where $u^t$ denotes transposition of the vector $u$, so that $\rho(A)$
is the the largest eigenvalue of $A$. Then, for any $n\times n$ matrix
$B$,
\[
\operatorname{tr}(AB) \leq\rho(A) \operatorname{tr}(B) .
\]
If $\phi$ is bounded away from zero, say $\phi(x) \geq a >0$ for all
$x\in[-\pi,\pi]$, then
\[
\rho(M_n^{-1} (\phi)) \leq a^{-1} .
\]
By the Parseval--Bessel inequality,
\[
\operatorname{tr}(\{M_n(\phi)\}^2) \leq n \int_{-\pi}^\pi\phi^2(x) \,\dd
x .
\]
Thus, if $g$ is bounded below, then $I_n(\theta)$ is bounded by some
constant times
\[
n \int_{-\pi}^\pi k^2\bigl((t-t_0)/h\bigr) \,\dd t = nh \int k^2(t) \,\dd t .
\]

In order to get an expression for $\tilde{I}({\lambda})$, let
${\lambda}_{0}$
be an arbitrary density on $(-1,1)$, and define the prior density on
${\Theta}=(-\delta h,\delta h)$ as ${\lambda}({\theta})=\frac
{1}{\delta
h}{\lambda}_{0}(\frac{{\theta}}{\delta h})$. Then
\[
\tilde{I}({\lambda}) = \int_{-\delta h}^{\delta h}
\frac{({\lambda}'({\theta}))^{2}}{{\lambda}({\theta})} \,\dd{\theta}
= \frac{1}{\delta^2 h^{2}} \int_{-1}^{1}
\frac{{\lambda}_{0}'(u)^{2}}{{\lambda}_{0}(u)} \,\dd u =
\frac{I_{0}}{\delta^{2}h^2} .
\]
Finally, plugging the previous bounds into (\ref{eq:vt})
yields, for large enough $n$,
\[
\sup_{g\in U} \sup_{L\in\mathcal{L}_g}{\mathbb E}_{L}\bigl(T_{n}(t_{0}) -
g(t_{0})\bigr)^{2} \geq\frac{1}{nhc_{3} + {I_{0}}{{\delta}^{-2}h^{-2}}},
\]
which, if $h=n^{-1/3}$, becomes
\[
\sup_{g\in U} \sup_{L\in\mathcal{L}_g} {\mathbb E}_{L} [ \{
T_{n}(t_{0}) - g(t_{0}) \}^{2} ] \geq c_4n^{-2/3} ,
\]
for some positive constant $c_4$. This completes the proof of
Theorem \ref{theo:lowerbound}.
\end{pf}

%s4 ###
\section{Limit distribution results}
\label{sec:limit-distr}
We next derive the limit distributions for $\hat{f}_n$ and
$\tilde{f}_n$ under general assumptions. The main tools used are local
limit distributions for the rescaled empirical spectral\vspace*{2pt} distribution
function $F_n$ and empirical log-spectral distribution function
$\tilde{F}_n$, respectively, as well as
maximal bounds for the rescaled processes. These will be coupled with
smoothness results for the least
concave majorant map established
% in \citet{anevskihossjer2006}. The next result states
%pointwise limit
% distribution results for the greatest convex minorant of $J_n$ and its
% derivative, proved in a similar setting
in \citet{anevskihossjer2006}, Theorems 1 and~2.
The proofs are postponed to Section \ref{section:proofs}.
%
% In the next subsections, we apply Theorem \ref{theo:anevskihossjer} in
% the two cases $J_n(t)=F_n(t)$ and $J_n(t)=\tilde{F}_n(t)$.

%s4.1 ###
\subsection{The limit distribution for the estimator $\hat{f}_n$}

\begin{assumption}
The process $\{X_i, i\in{\mathbb Z}\}$ is linear with respect to
an i.i.d. sequence $\{\varepsilon_i, i\in{\mathbb Z}\}$ with zero
mean and unit variance, that is,
%
%e10 ###
\begin{equation} \label{eq:linproc}
X_k = \sum_{j=0}^{\infty} a_j \varepsilon_{k-j} ,
\end{equation}
where the sequence $\{a_j\}$ satisfies
%
%e11 ###
\begin{equation}\label{eq:spectdens}
\sum_{j=1}^{\infty} (j^{1/2}|a_j| + j^{3/2} a_j^2) < \infty.
\end{equation}
\end{assumption}
\begin{remark}
Condition (\ref{eq:spectdens}) is needed to deal with remainder
terms and apply the results of \citet{mikoschnorvaisa1997} and
\citet{brockwelldavis1991}. It is implied, for instance, by the
simpler condition
%
%e12 ###
\begin{equation}
\sum_{j=1}^{\infty} j^{3/4}|a_j| < \infty.
\end{equation}
It is satisfied by most usual linear time series such as causal
invertible ARMA processes.
\end{remark}

The spectral density of the process $\{X_i\}$ is given by
\[
f(u) = \frac{1}{2\pi} \Biggl|\sum_{j=0}^{\infty} a_j
e^{ij u} \Biggr|^2 .
\]
Unfortunately, there is no explicit condition on the coefficients
$a_j$ that implies monotonicity of $f$, but the coefficients $a_j$ are
not of primary interest here.

The limiting distribution of the estimator will be expressed in terms
of the
so-called Chernoff distribution, that is, the law of a random variable
$\zeta$ defined by $\zeta= \arg\max_{s\in\mathbb R} \{W(s)-s^2\}$,
where $W$ is
a standard two-sided Brownian motion. See \citet{groeneboomwellner2001} for
details about this distribution.
%
% It can be shown by applying the scaling properties of the Brownian
% motion that the distribution of $T(y)'(0)$ is the same as that of
% $2\{-\pi f^2(t_0) f'(t_0)\}^{1/3} \zeta$, where the law of $\zeta$
% is the so-called Chernoff's distribution,
%
% Define the process $y$ by
% \begin{gather*}
% y(s) = \frac12 f'(t_0) s^2 + \sqrt{2\pi} f(t_0) W(s) ,
% \end{gather*}
%
\begin{theorem} \label{theo:lineaire-monotone}
Let $\{X_i\}$ be a linear process such that (\ref{eq:linproc})
and (\ref{eq:spectdens}) hold and $\mathbb{E}[\varepsilon_0^8] <
\infty$. Assume that its spectral density $f$ belongs to $\mathcal{
F}$. Assume $f'(t_0)<0$ at the fixed point $t_0$. Then, as
$n\rightarrow\infty$,
\[
n^{1/3} \bigl(\hat{f}_n(t_0) - f(t_0)\bigr) \stackrel{\mathcal{L}}{\rightarrow
} 2\{-\pi
f^2(t_0) f'(t_0)\}^{1/3} \zeta.
\]
\end{theorem}

%s4.2 ###
\subsection{The limit distributions for the estimator $\tilde{f}_n$}

In this section, in order to deal with the technicalities of the
log-periodogram, we make the following assumption.
\begin{assumption} \label{hypo:gaussien}
The process $\{X_k\}$ is Gaussian. Its spectral density $f$ is
monotone on $(0,\pi]$ and can be expressed as
\[
f(x) = |1-{e}^{i x}|^{-2d} f^*(x) ,
\]
with $|d| < 1/2$ and $f^*$ is bounded above and away from zero and
there exists a constant $C$ such that for all $x,y\in(0,\pi]$,
\[
|f(x) - f(y)| \leq C \frac{|x-y|}{x \wedge y} .
\]
\end{assumption}
\begin{remark}
This condition is usual in the long memory literature. Similar
conditions are assumed in \citet{robinson1995l}, Assumption 2,
\citet{moulinessoulier1999}, Assumption 2,
\citet{soulier2001}, Assumption 1 (with a typo). It is used to
derive covariance bounds for the discrete Fourier transform
ordinates of the process, which yield covariance bounds for
nonlinear functionals of the periodogram ordinates in the Gaussian
case. It is satisfied by usual long memory processes such as causal
invertible $\operatorname{ARFIMA}(p,d,q)$ processes with possibly an additive
independent white noise or AR(1) process.
\end{remark}

Recall that
\[
\tilde f_n = \exp\mathop{\arg\min}_{f\in\mathcal F} \int_0^\pi\{\log f(s) -
\log I_n(s) + \gamma\}^2 \,\dd s,
\]
where $\gamma$ is Euler's constant and $I_n$ is the periodogram,
defined here as a step function
\[
I_n(t) = I_n(2\pi[nt/2\pi]/n) = \frac{2\pi}n \Biggl| \sum_{k=1}^n X_k
{e}^{i 2k\pi[nt/2\pi]/n} \Biggr|^2 .
\]
%
% Denote
% \begin{gather*}
% \tilde y (s) = \frac12 \frac{f'(t_0)}{f(t_0)} s^2 + \sqrt{2\pi^4/3}
% W(s) ,
% \end{gather*}
% where $W$ is a standard two sided Brownian motion.
%
\begin{theorem} \label{theo:gaussien-lrd-monotone}
Let $\{X_i\}$ be a Gaussian process that satisfies
Assumption \ref{hypo:gaussien}. Assume $f'(t_0)<0$ at the fixed
point $t_0 \in(0,\pi)$. Then, as $n\rightarrow\infty$,
\[
n^{1/3} \{\log\tilde{f}_n(t_0) - \log f(t_0)\} \stackrel{\mathcal
{L}}{\rightarrow}
2 \biggl(\frac{-\pi^4 f'(t_0)}{3f(t_0)} \biggr)^{1/3} \zeta.
\]
\end{theorem}
\begin{corollary}
Under the assumptions of Theorem \ref{theo:gaussien-lrd-monotone},
\[
n^{1/3} \{\tilde{f}_n(t_0) - f(t_0)\} \stackrel{\mathcal
{L}}{\rightarrow}
2\{-\pi^4f^2(t_0)f'(t_0)/3\}^{1/3} \zeta.
\]
\end{corollary}
\begin{remark}
This is the same limiting distribution as in
Theorem \ref{theo:lineaire-monotone}, up to the constant
$3^{-1/3}\pi>1$. Thus the estimator $\tilde f_n$ is less efficient
than the
estimator $\hat f_n$, but the interest of $\tilde f_n$ is to be used when
long memory is suspected, that is, the spectral density exhibits a
singularity at
zero, and the assumptions of Theorem \ref{theo:lineaire-monotone} are not
satisfied.
\end{remark}

%s5 ###
\section{Simulations and finite sample behavior of estimators}
\label{sec:simulations}

In this section we apply the nonparametric methods on simulated time
series data of sums of parametric models. The algorithms used for the
calculation of $\hat{f}_n$ and
$\tilde{f}_n$ are the discrete versions of the estimators
%
%f1 ###
\begin{figure}[b]

\includegraphics{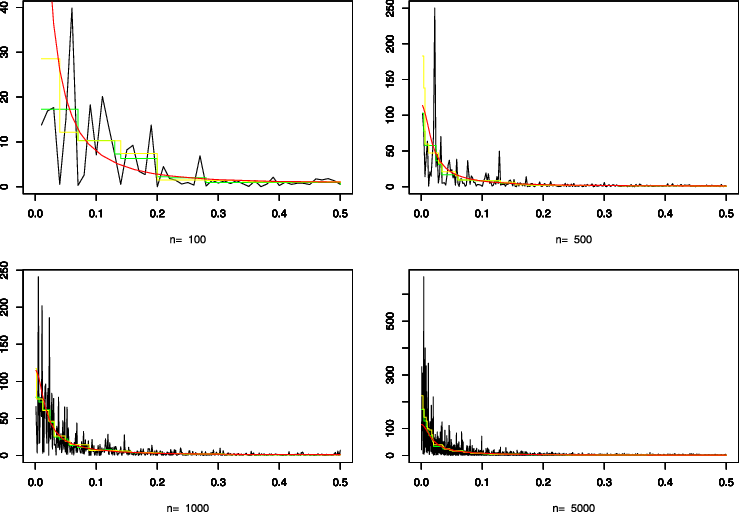}

\caption{The spectral density (red), the periodogram (black), the
estimates $\hat{f}_n$ (green) and $\tilde{f}_n$ (yellow), for
$n=100,500,1\mbox{,}000$ and $5\mbox{,}000$ data points, for Example \protect\ref{example:3ar}.}
\label{fig:3ar}
\end{figure}
$\hat{f},\tilde{f}_n$, that are obtained by doing isotonic regression
of the data $\{(\lambda_k,I_n(\lambda_k)) , k=1,\ldots,[(n-1)/2]\}$
where $\lambda_k=2\pi k/n$. For instance, the discrete version
$\hat{f}_n^d$ of $\hat{f}_n$ is calculated as
\[
\hat{f}_n^d=\mathop{\arg\min}_{z\in\mathcal{F}}\sum_{k=1}^n\bigl(I_n(\lambda
_k)-z(\lambda_k)\bigr)^2.
\]
Note that the limit distribution for $\tilde{f}_n$ is stated for the
discrete version $\tilde{f}_n^d$. The simulations were done in R,
using the ``fracdiff'' package. The code is available from the
corresponding author upon request.
\begin{example}
\label{example:3ar}
The first example consists of sums of several AR(1) processes. Let
$\{X_k\}$ be a stationary AR(1) process, that is, for all $k\in\mathbb
Z$,
\[
X_k=a X_{k-1}+\varepsilon_k ,
\]
with $|a|<1$. This process has spectral density function $f(\lambda) =
(2\pi)^{-1}\sigma^2|1-ae^{i\lambda}|^{-2}$ for $-\pi
\leq\lambda\leq
\pi$, with $\sigma^2=\operatorname{var}(\varepsilon_1^2)$ and and thus
$f$ is
decreasing on $[0,\pi]$. If $X^{(1)},\ldots,X^{(p)}$ are independent
AR(1) processes with coefficients $a_j$ such that $|a_j|<
1,j=1,\ldots,p$, and we define the process $X$ by
\[
X_k = \sum_{j=1}^p X_k^{(j)},
\]
then $X$ has spectral\vspace*{1pt} density $f(\lambda) = (2\pi)^{-1} \sum_{j=1}^p
\sigma_j^2
|1+a_je^{i\lambda}|^{-2}$ which is decreasing on $[0,\pi]$, since it
is a sum of decreasing functions. Assuming that we do not know how
many AR(1) processes are summed, we have a nonparametric problem:
estimate a monotone spectral density. Figure \ref{fig:3ar} shows a
%
%f2 ###
\begin{figure}

\includegraphics{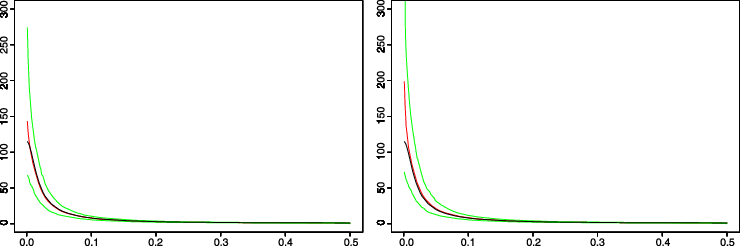}

\caption{Left plot: spectral density (black), pointwise mean of
estimates $\hat{f}_n$ (red) and
95\% confidence intervals (green). Right plot: spectral density
(black), pointwise mean of the
estimates $\tilde{f}_n$ (red) and 95\% confidence intervals (green),
for $n=1\mbox{,}000$ data points, for
Example \protect\ref{example:3ar}.}
\label{fig:3ar-mean}
\end{figure}
plot of the periodogram, the true spectral density and the
nonparametric estimators $\hat{f}_n$ and $\tilde{f}_n$ for simulated
data from a sum of three independent AR(1) processes with
$a_1=0.5,a_2=0.7,a_3=0.9$. Figure \ref{fig:3ar-mean} shows the
pointwise means and 95\% confidence intervals of $\hat{f}_n$ and
$\tilde{f}_n$ for $1\mbox{,}000$ realizations.
\end{example}
\begin{example}
\label{example:arfima}
The second example is a sum of an $\operatorname{ARFIMA}(0,d,0)$
process and an AR(1)
process. Let $X^{(1)}$ be an $\operatorname{ARFIMA}(0,d,0)$-process
with $0<d<1/2$.
This has a spectral density $(2\pi)^{-1}\sigma_1^2|1-
e^{i\lambda}|^{-2d}$. If we add an independent AR(1)-process
$X^{(2)}$ with coefficient $|a|<1$ the resulting process\vspace*{1pt}
$X=X^{(1)}+X^{(2)}$ will have spectral density $f(\lambda) =
(2\pi)^{-1}\sigma_1^2|1 - e^{i \lambda}|^{-2d} +
(2\pi)^{-1}\sigma_2^2|1 -
a e^{i \lambda}|^{-2}$ on $[0,\pi]$, and thus the
resulting spectral density $f$ will be a monotone function on
$[0,\pi]$. As above, if an unknown number of independent
processes is added we have a nonparametric estimation problem. Figure
\ref{fig:arfima} shows
a plot of the periodogram, the true spectral density and the
nonparametric estimators $\hat{f}_n$ and $\tilde{f}_n$ for simulated
time series data from a sum of an $\operatorname
{ARFIMA}(0,d,0)$-process with $d=0.2$
and an AR(1)-process with $a=0.5$. Figure \ref{fig:arfima-mean} shows
the pointwise means and 95\% confidence intervals of $\hat{f}_n$ and
$\tilde{f}_n$ for $1\mbox{,}000$ realizations.
\end{example}

Table \ref{tabell1} shows mean square root of sum of squares errors
(comparing with the true function), calculated on $1\mbox{,}000$ simulated
samples of the times series of Example \ref{example:3ar}.
Table \ref{tabell2} shows the analog values for
Example \ref{example:arfima}.

%
%f3 ###
\begin{figure}

\includegraphics{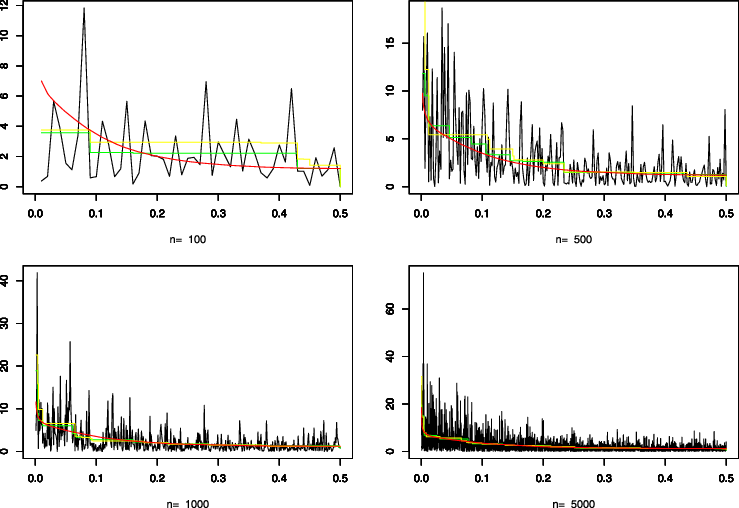}

\caption{The spectral density (red), the periodogram (black), the
estimates $\hat{f}_n$ (green) and $\tilde{f}_n$ (yellow), for
$n=100,500,1\mbox{,}000$ and $5\mbox{,}000$ data points, for
Example \protect\ref{example:arfima}.}
\label{fig:arfima}
\end{figure}

%
%f4 ###
\begin{figure}%[b]

\includegraphics{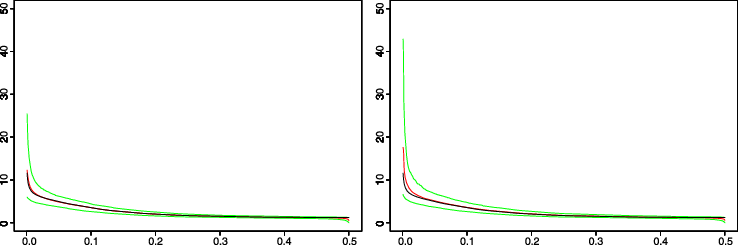}

\caption{Left plot: spectral density (black), pointwise mean of
estimates $\hat{f}_n$ (red) and
95\% confidence intervals (green). Right plot: spectral density
(black), pointwise mean of the
estimates $\tilde{f}_n$ (red) and 95\% confidence intervals (green),
for $n=1\mbox{,}000$ data points, for
Example \protect\ref{example:arfima}.}
\label{fig:arfima-mean}
\end{figure}
%

%t1 ###
\begin{table}
\caption{MISE values for Example \protect\ref{example:3ar}}
\label{tabell1}
\begin{tabular*}{\tablewidth}{@{\extracolsep{\fill}}lcd{2.2}d{2.2}d{2.2}@{}}
\hline
\textbf{MISE} & \multicolumn{1}{c}{$\bolds{n=100}$} &
\multicolumn{1}{c}{$\bolds{n=500}$}
& \multicolumn{1}{c}{$\bolds{n=1\mbox{,}000}$} & \multicolumn{1}{c@{}}{$\bolds{n=5\mbox{,}000}$} \\
\hline
$I_n$ & 9.59 & 12.96 & 13.67 & 14.25\\
$\hat{f}_n$ & 6.38 & 5.48 & 4.76 & 2.95\\
$\tilde{f}_n$ & 9.11 & 8.52 & 7.27 & 4.26\\
\hline
\end{tabular*}
\end{table}

%t2 ###
\begin{table}%[b]
\caption{MISE values for Example \protect\ref{example:arfima}}
\label{tabell2}
\begin{tabular*}{\tablewidth}{@{\extracolsep{\fill}}ld{1.3}d{1.3}d{1.3}d{1.3}@{}}
\hline
\textbf{MISE} & \multicolumn{1}{c}{$\bolds{n=100}$} & \multicolumn{1}{c}{$\bolds{n=500}$}
& \multicolumn{1}{c}{$\bolds{n=1\mbox{,}000}$} & \multicolumn{1}{c@{}}{$\bolds{n=5\mbox{,}000}$} \\
\hline
$I_n$ & 1.80 & 1.99 & 2.02 & 2.07 \\
$\hat{f}_n$ & 0.710 & 0.520& 0.432 & 0.305 \\
$\tilde{f}_n$ & 1.12 & 0.803 & 0.659 & 0.472\\
\hline
\end{tabular*}
\end{table}

Both estimators $\hat{f}_n$ and $\tilde{f}_n$ seem to
have good finite sample properties. As indicated by the theory, $\tilde
{f}_n$ seems to be less efficient than $\hat{f}_n$.

%s6 ###
\section{\texorpdfstring{Proofs of Theorems \protect\ref{theo:lineaire-monotone}
and \protect\ref{theo:gaussien-lrd-monotone}}{Proofs of Theorems 5 and 6}}
\label{section:proofs}

% We start with a review of the general limit distribution results
%which will
% be applied to the monotone spectral density estimation problem.
Let $J_{n}$ be the integral of the generic preliminary estimator of the
spectral density, that is the integral of $I_n$ or of $\log(I_n)$, let $K$
denote $F$ or the primitive of $\log f$, respectively, and write
%
%e13 ###
\begin{equation} \label{eq:PaAPROC}
J_{n}(t) = K(t) + v_{n}(t) .
\end{equation}
Let $d_{n}\downarrow0$ be a deterministic sequence, and define the rescaled
process and rescaled centering
%
%e15 ###
%e14 ###
\begin{eqnarray} \label{eq:def-rescaled}
\tilde{v}_{n}(s;t_{0}) &=& d_{n}^{-2} \{v_{n}(t_{0} + s d_{n}) -
v_{n}(t_0)\} , \\
g_{n}(s) &=& d_{n}^{-2} \{K(t_{0}+sd_{n})-K(t_0) - K'(t_0)
d_n s \} .
\end{eqnarray}
Consider the following conditions:
\begin{enumerate}[(AH1)]
\item[(AH1)]\hypertarget{item:conv-rescaled} There exists a stochastic process
$\tilde{v}(\cdot;t_{0})$ such that
%
%e16 ###
\begin{equation}\label{eq:PaAprocconv}
\tilde{v}_{n}(\cdot;t_{0}) \stackrel{\mathcal{L}}{\rightarrow}
\tilde{v}(\cdot;t_{0}) ,
\end{equation}
in $D(-\infty,\infty)$, endowed with the topology generated by the supnorm
metric on compact intervals, as $n\rightarrow\infty$.
\item[(AH2)]\hypertarget{item:cone} For each $\varepsilon,\delta>0$ there is a
finite $\tau$ such that
%
%e18 ###
%e17 ###
\begin{eqnarray}
\label{eq:PaAprop1}
\limsup_{n\rightarrow\infty} P \biggl( \sup_{|s|\geq\tau} \biggl|
\frac{\tilde{v}_{n}(s;t_0)}{g_{n}(s)} \biggr| > \varepsilon\biggr) & < &
\delta, \\
\label{eq:coneBM}
P \biggl( \sup_{|s|\geq\tau} \biggl| \frac{\tilde{v}(s;t_{0})}{s^{2}} \biggr|
\varepsilon\biggr) & < & \delta.
\end{eqnarray}
\item[(AH3)]\hypertarget{item:parabole} There is a constant $A<0$ such that
for each $c>0$,
%
%e19 ###
\begin{equation}\label{eq:PaAbiasconv}
\lim_{n\to\infty} \sup_{|s|\leq c}|g_{n}(s)-As^{2}| = 0 ;
\end{equation}
\item[(AH4)]\hypertarget{item:parabole2} For each $a\in{\mathbb R}$ and
$c,\varepsilon>0$
%
%e20 ###
\begin{equation} \label{eq:PaApanik1}
P\bigl( \tilde v(s;t_0)(s) - \tilde v(0;t_0) + As^2 - as \geq\varepsilon
|s| \mbox{ for all } s \in[-c,c]\bigr) = 0 .
\end{equation}
\end{enumerate}
If there exists a sequence $d_n$ such that these four conditions hold, then,
defining the process $y$ by $y(s) = \tilde v(s;t_0)+As^2$, by
Anevski and H{\"o}ssjer [(\citeyear{anevskihossjer2006}), Theorems 1
and 2] as $n\rightarrow\infty$, it holds that
%
%e21 ###
\begin{equation} \label{eq:conv-slope}
d_{n}^{-1} \{ T(J_{n})'(t_{0}) - K'(t_{0}) \} \stackrel{\mathcal
{L}}{\rightarrow} T(y)' (0) ,
\end{equation}
where $T(y)'(0)$ denotes the slope at zero of the smallest concave
majorant of
$y$.

% \begin{remark}
% The map $T$ possesses certain smoothness and truncation properties
% that imply limit distributions for $T(J_n)$ under
% Assumptions (\ref{eq:PaAprocconv})-(\ref{eq:PaAbiasconv}). The map
% $h\mapsto T(h)(0)$ is continuous but the map $h\mapsto T(h)'(0)$ is
% not continuous, so Assumption (\ref{eq:PaApanik1}) is essential for
% obtaining limit distributions for $T(J_n)'$; it ensures that the map
% $h\to T(h)'(0)$ is continuous at the limit point $y$,
% cf. \citet[Proposition 2]{anevskihossjer2006}. The supnorm
% topology is used to ensure the continuity properties for the map
% $T$; $T$ is not continuous e.g. in the Skorokhod topology. We use
% the $\sigma$-algebra generated by the open balls to avoid
% measurability issues. The assumption in Theorem 5 can be weakened to
% demanding weak convergence in the Skorokhod $J_1$ topology when the
% limit process $\tilde{v}(\cdot;t_0)$ is continuous, since then $J_1$
% convergence implies uniform convergence on compact sets.
% \end{remark}

%s6.1 ###
\subsection{\texorpdfstring{Proof of Theorem
\protect\ref{theo:lineaire-monotone}}{Proof of Theorem 5}}
\label{sec:proof-theo-linmono}
The proof consists in checking conditions
\hyperlink{item:conv-rescaled}{(AH1)}--\hyperlink{item:parabole2}{(AH4)}
with $J_n = F_n$ and $K = F$.
\begin{enumerate}[-]
\item[-] It is proved in Lemma \ref{lemma:rescaled} below
that (\ref{eq:PaAprocconv}) holds with $d_n = n^{1/3}$ and $\tilde
v(\cdot;t_0)$ the standard two-sided Brownian motion times $\sqrt{\pi^2/6}$.
\item[-] If $f'(t_0)<0$, then (\ref{eq:PaAbiasconv}) holds with $A =
\frac12 f'(t_0)$ and $d_n\downarrow0$ an arbitrary deterministic
sequence.
\item[-] Lemma \ref{lem:truncation} shows that (\ref{eq:PaAprop1}) holds
and the law of iterated logarithm yields that (\ref{eq:coneBM})
holds for the two-sided Brownian motion.
\item[-] Finally, (\ref{eq:PaApanik1}) also holds for the two-sided
Brownian motion.
\end{enumerate}
Thus (\ref{eq:conv-slope}) holds with the process $y$ defined by
\[
y(s) = \tfrac12 f'(t_0) s^2 + \sqrt{2\pi} f(t_0) W(s) .
\]
The scaling property of the Brownian motion yields the representation of
$T(y)'(0)$ in terms of Chernoff's distribution.

% Recall that $\tilde v_n$ is defined in (\ref{eq:def-rescaled}) as
% \begin{align*}
% \tilde v_n(s;t_0) & = d_n^{-2} \int_{t_0}^{t_0+d_ns} (I_n(u)-f(u))
% du .
% \end{align*}
%
\begin{lem} \label{lemma:rescaled} Assume the process $\{X_n\}$ is
given by
(\ref{eq:linproc}), that (\ref{eq:spectdens}) holds and that
$\esp[\varepsilon_0^8]<\infty$. If $d_n=n^{-1/3}$, then the sequence of
processes $\tilde v_n(\cdot;t_0)$ defined in (\ref{eq:def-rescaled})
converges weakly in $C([-c,c])$ to $\sqrt{2\pi} f(t_0) W$ where $W$
is a
standard two-sided Brownian motion.
\end{lem}
\begin{pf}
For clarity, we omit $t_0$ in the notation. Write
\[
\tilde v_n(s) = \tilde v_n^\varepsilon(s) + R_n(s)
\]
with
%
%e23 ###
%e22 ###
\begin{eqnarray}\quad
\label{eq:defvnepsilon}
\tilde{v}_n^{(\varepsilon)}(s) &=& d_n^{-2} \int_{t_0}^{t_0+d_ns} f(u)
\bigl\{I_n^{(\varepsilon)}(u)-1\bigr\} \,\dd u ,\nonumber\\[-8pt]\\[-8pt]
I_n^{(\varepsilon)}(u) &=& \frac{1}{n} \Biggl| \sum_{k=1}^n \varepsilon_k
e^{i ku} \Biggr|^2 , \nonumber\\
\label{eq:remainder}
R_n &=& d_n^{-2} \int_{t_0}^{t_0+d_ns} r_n(u) \,\dd u ,\qquad
r_n(u) = I_n(u)-f(u)I_n^{(\varepsilon)}(u) .
\end{eqnarray}
Note that $(2\pi)^{-1} I_n^\varepsilon$ is the periodogram for the white
noise sequence $\{\varepsilon_i\}$. We first treat the remainder term
$R_n$. Denote $\mathcal{G}=\{g\dvtx\int_{-\pi}^{\pi} g^2(u) f^2(u)\,
du<\infty\}$.
Equation (5.11) (with a typo in the normalization) in
\citet{mikoschnorvaisa1997} states that if (\ref
{eq:spectdens}) and
$\mathbb{E}[\varepsilon_0^8]<\infty$ hold, then
%
%e24 ###
\begin{equation} \label{eq:miknor}
\sqrt{n}\sup_{g\in\mathcal{G}}\int_{-\pi}^{\pi} g(x)r_n(x) \,\dd x =
o_P(1) .
\end{equation}
Define the set $\tilde{\mathcal{G}} = \{k_{n}(\cdot,s)f\dvtx n \in
{\mathbf
N}, s \in[-c,c]\}$. Since $f$ is bounded, we have that $\int
k_{n}^2(u,s) f^2(u) \,\dd u < \infty$, so $\tilde\mathcal{G} \subset
\mathcal{G}$ and we can apply (\ref{eq:miknor}) on $\tilde\mathcal
{G}$, which shows that $R_n$ converges uniformly (over $s\in[-c,c]$)
to zero.
We next show that
%
%e25 ###
\begin{equation} \label{eq:rescalediid}
\tilde{v}_n^{(\varepsilon)}(s) \stackrel{\mathcal{L}}{\rightarrow}
\sqrt{2 \pi}
f(t_0) W(s) ,
\end{equation}
as $n\rightarrow\infty$, on $C({\mathbb R}) $, where $W$ is a standard
two-sided Brownian motion. Since $\{\varepsilon_k\}$ is a white noise
sequence, we set $t_0=0$ without loss of generality. Straightforward
algebra yields % The centered
% periodogram can be expressed as
% \begin{eqnarray*}
% I_n^{(\varepsilon)}(u)-1 & = & \frac{1}{n} \sum_{k=1}^n (
% \frac{2}{n} \sum_{k=2}^n \sum_{j=1}^{k-1} \cos\{(k-j)u\}
% \varepsilon_j \varepsilon_k,
% \end{eqnarray*}
% so that
%
%e26 ###
\begin{equation}\label{eq:44}
\tilde{v}_n^{(\varepsilon)}(s) = d_n^{-2} \{\hat{\gamma}_n(0)-1\}
F(d_ns) + S_n(s)
\end{equation}
with
\begin{eqnarray*}
\hat{\gamma}_n(0) &=& n^{-1} \sum_{j=1}^{n} \varepsilon_j^2 ,\qquad
S_n(s) = \sum_{k=2}^n C_k(s) \varepsilon_k , \\
C_k(s) &=& d_n^{3/2} \sum_{j=1}^{k-1} \alpha_j(s) \varepsilon_{k-j}
,\qquad
\alpha_j(s) = d_n^{-1/2} \int_{-d_n s}^{d_n s} f(u) {
e}^{ij u } \,\dd u .
\end{eqnarray*}
%
%We first prove that the first term in (\ref{eq:44}) is negligible.
Since $\{\varepsilon_j\}$ is a white noise sequence with finite fourth
moment, it is easily checked that
%
%e27 ###
\begin{eqnarray}\label{eq:varhatgamma0}
n\operatorname{var}(\hat{\gamma}_n(0)) &=& \operatorname{var}(\varepsilon_0^2) ,
\nonumber\\[-8pt]\\[-8pt]
\sup_{s\in[-c,c]} d_n^{-2} \int_0^{d_ns} f(u) \,\dd u\,
|\hat{\gamma}_n(0)-1| &=& O_P(d_n^{-1}n^{-1/2}) = O_P\bigl(\sqrt{d_n}\bigr)
\nonumber
\end{eqnarray}
so that the first term in (\ref{eq:44}) is negligible. It remains to prove
that the sequence of processes $S_n$ converges weakly to a standard Brownian
motion. We prove the convergence of finite dimension distribution by
application of the Martingale central limit theorem [cf., e.g.,
\citet{hallheyde1980}, Corollary 3.1]. It is sufficient to
check the following
conditions:
%
%e29 ###
%e28 ###
\begin{eqnarray}
\label{eq:variance}
\lim_{n\to\infty} \sum_{k=2}^n \esp[C_k^2(s)] &=& 2 \pi f^2(0) s ,
\\
\label{eq:negligibility}
\lim_{n\to\infty} \sum_{k=2}^n \esp[C_k^4(s)] &=& 0 .
\end{eqnarray}
By the Parseval--Bessel identity, we have
\[
\sum_{j=-\infty}^\infty\alpha_j^2(s) = 2 \pi d_n^{-1} \int_{-d_n
s}^{d_n s} f^2(u) \,\dd u \sim4 \pi f^2(0) s .
\]
Since $\alpha_0(s) \sim2f(0) \sqrt{d_n}$, this implies that
\[
\sum_{k=2}^n \esp[C_k^2(s)] % & = n^{-1} \sum_{k=2}^n \sum_{j=1}^{k-1}
% \alpha_j^2(s)
= \sum_{j=1}^{n-1} (1-j/n) \alpha_j^2(s) \sim\sum_{j=1}^\infty
\alpha_j^2(s)\sim2 \pi f^2(0) s .
\]
This proves condition (\ref{eq:variance}). For the asymptotic
negligibility condition (\ref{eq:negligibility}), we use Rosenthal's
inequality [cf. \citet{hallheyde1980}, Theorem 2.12],
\[
E[C_k^4] \leq cst n^{-2} \sum_{j=1}^{k-1} \alpha_j^4(s) +
cst n^{-2} \Biggl( \sum_{j=1}^{k-1} \alpha_j^2(s) \Biggr)^2 = O(n^{-2}),
\]
implying $\sum_{k=1}^n \esp[C_k^4(s)] = O(n^{-1})$, which
proves (\ref{eq:negligibility}). To prove tightness, we compute the fourth
moment of the increments of $S_n$. Write
\[
S_n(s) - S_n(s') = n^{-1/2} \sum_{k=1}^n \sum_{j=1}^{k-1}
\alpha_j(s,s') \varepsilon_{k-j} \varepsilon_k ,
\]
with
\[
\alpha_j(s,s') = d_n^{-1/2} \int_{d_ns'}^{d_ns} f(u) {e}^{i j
u} \,\dd u + d_n^{-1/2} \int_{-d_ns}^{-d_ns'} f(u) {e}^{
i j
u} \,\dd u .
\]
By Parseval's inequality, it holds that
\[
\sum_{j=1}^n \alpha_j^2(s,s') \leq C |s-s'| .
\]
Applying again Rosenthal's inequality, we obtain that
$\esp[|S_n(s)-S_n(s')|^4]$ is bounded by a constant times
\[
n^{-1} \sum_{j=1}^n \alpha_j^4(s,s') + \Biggl(\sum_{j=1}^n
\alpha_j^2(s,s') \Biggr)^2 \leq C |s-s'|^2 .
\]
Applying [Billingsley (\citeyear{billingsley1968}), Theorem 15.6]
concludes the proof of
tightness.
\end{pf}
\begin{lem} \label{lem:truncation}
For any $\delta>0$ and any $\kappa>0$, there exists $\tau$ such that
%
%e30 ###
\begin{equation} \label{eq:cone}
\limsup_{n\to\infty} \pr\biggl( \sup_{|s| \geq\tau} \frac{|\tilde
v_n(s)|}{|s|} > \kappa\biggr) \leq\delta.
\end{equation}
\end{lem}
\begin{pf}
Without loss of generality, we can assume that $f(t_0)=1$. Recall
that $\tilde v_n = \tilde v_n^{(\varepsilon)} + R_n$ and
$\tilde{v}_n^{(\varepsilon)}(s) = F(d_ns) \zeta_n + S_n(s)$, where
$\tilde v_n^{(\varepsilon)}$ and $R_n$ are defined
in (\ref{eq:defvnepsilon}) and (\ref{eq:remainder}), $\zeta_n =
d_n^{-2} (\hat\gamma_n(0)-1)$ and $S_n$ is defined
in (\ref{eq:44}). Then
\begin{eqnarray*}
\pr\biggl( \sup_{s \geq\tau} \frac{|\tilde v_n(s)|}{s} > \kappa
\biggr) & \leq & \pr\Bigl( \sup_{s\geq\tau} |\zeta_n| F(d_ns)/s> \kappa/3 \Bigr)
+ \pr\biggl( \sup_{s\geq\tau} \frac{|S_n(s)|}{s} > \kappa/3 \biggr)\\
&&{}  +
\pr\biggl( \sup_{s\geq\tau} \frac{|R_n(s)|}{s} > \kappa/3 \biggr) .
\end{eqnarray*}
The spectral\vspace*{1pt} density is bounded, so $F(d_ns)/s \leq Cd_n$ for all $s$.
Since $\operatorname{var}(\zeta_n) = O(d_n^{-1})$, by (\ref
{eq:varhatgamma0}) and
the Bienayme--Chebyshev inequality, we get
\[
\pr\Bigl( \sup_{s\geq\tau} |\zeta_n| F(d_ns)/s> \kappa\Bigr)
\leq O(d_n^{-1}d_n^{2}) = O(d_n) .
\]
Let $\{s_j, j\geq0\}$ be an increasing sequence. Then we have the bound
\begin{eqnarray*}
\pr\biggl( \sup_{s \geq s_0} \frac{|S_n(s)|}{s} > \kappa
\biggr) & \leq & \sum_{j=0}^\infty\pr\bigl(|S_n(s_j)| > \kappa s_j \bigr) \\
&&{} + \sum_{j=1}^\infty\pr\Bigl( {\sup_{s_{j-1} \leq s \leq s_j}}
|S_n(s) -S_n(s_{j-1})| > \kappa s_{j-1} \Bigr) .
\end{eqnarray*}
From (\ref{eq:variance}), we know that $\operatorname{var}(S_n(s)) = O(s)$. Thus
\[
\sum_{j=0}^\infty\pr\bigl(|S_n(s_j)| >\kappa s_j\bigr) \leq cst
\kappa^{-2} \sum_{j=0} s_j^{-1} .
\]
Thus if the series $s_j^{-1}$ is summable, this sum can be made
arbitrarily small by choosing $s_0$ large enough. It was shown in the
proof of Lemma \ref{lemma:rescaled} that
\[
\esp[|S_n(s)-S_n(s')|^4] \leq C |s-s'|^2 .
\]
By \citet{billingsley1968}, Theorem 15.6 [or more specifically
\citet{ledouxtalagrand1991}, Theorem 11.1], this implies that
\[
\pr\Bigl( {\sup_{s_{j-1} \leq s \leq s_j}} |S_n(s) - S_n(s_{j-1})| >
\kappa s_{j-1} \Bigr) \leq\frac{C(s_j-s_{j-1})^2}{\kappa^2
s^2_{j-1}} .
\]
Thus choosing $s_j = (s_0+j)^{\rho}$ for some $\rho>1$ implies that
the series is convergent and
\[
\pr\biggl( \sup_{s \geq s_0} \frac{|S_n(s)|}{s} > \kappa\biggr) =
O(s_0^{-1}) ,
\]
which is arbitrarily small for large $s_0$.

To deal with the remainder term $R_n$, we prove that $\pr({\sup_{s\geq
s_0}}|R_n(s)|/s >s_0) = o_P(1)$ by the same method as that used for
$S_n$. Thus we only need to obtain a suitable bound for the
increments of $R_n$. By definition of $R_n$, we have, for $s<s'$,
\[
R_n(s')- R_n(s) = d_n^{-2} \int_{t_0+d_ns}^{t_0+d_ns'} f(u) r_n(u)
\,\dd u .
\]
Since $f$ is bounded, by H\"older's inequality, we get
\[
\esp[|R_n(s') - R_n(s)|^2] \leq\|f\|_\infty n (s'-s)
\int_{t_0+d_ns}^{t_0+d_ns'} \esp[r_n^2(u)] \,\dd u .
\]
Under (\ref{eq:spectdens}), it is known [see, e.g.,
\citet{brockwelldavis1991}, Theorem 10.3.1] that
\[
\esp[r_n^2(u)] \leq C n^{-1} .
\]
Hence,
\[
\esp[|R_n(s') - R_n(s)|^2] \leq C d_nn (s'-s)^2 .
\]
The rest of the proof is similar to the proof for $S_n$. This
concludes the proof of (\ref{eq:cone}).
\end{pf}

%s6.2 ###
\subsection{\texorpdfstring{Sketch of proof of Theorem \protect\ref
{theo:gaussien-lrd-monotone}}{Sketch of proof of Theorem 6}}
The proof consists in checking
conditions \hyperlink{item:conv-rescaled}{(AH1)}--\hyperlink{item:parabole2}{(AH4)} with
$J_n$ and $K_n$ now defined by $J_n(t) =
\int_0^t \{\log I_n(s) + \gamma\} \,\dd s $ and $K(t) = \int_0^t
\log
f(2\pi[ns/2\pi]/n) \,\dd s $.
Let $\lambda_k = 2 k \pi/n$ denote the so-called Fourier frequencies.
For $t\in[0,\pi]$, denote $k_n(t) = [nt/2\pi]$. Denote
\[
\xi_k = \log\{I_n(\lambda_k)/f(\lambda_k)\} + \gamma,
\]
where $\gamma$ is Euler's constant. Then
\[
v_n(t) = J_n(t) -K(t) = \frac{2\pi}n \sum_{j=1}^{k_n(t)} \xi_j +
\bigl(t-\lambda_{k_n(t)}\bigr) \xi_{k_n(t)} .
\]
The log-periodogram ordinates $\xi_j$ are not independent, but sums of
log-peri\-odogram ordinates, such as the one above, behave
asymptotically as sums of independent random variables with zero mean
and variance $\pi^2/6$ [cf. \citet{soulier2001}], and bounded moments
of all order. Thus, for $t_0\in(0,\pi)$, the process $\tilde
v_n(s;t_0) = d_n^{-2}\{v_n(t_0+d_ns) - v_n(t_0)\}$ with $d_n =
n^{-1/3}$ converges weakly in $D(-\infty,\infty)$ to the
two-sided Brownian motion with variance $2\pi^4/3$. It can be shown
by using the moment bounds of \citet{soulier2001}
that (\ref{eq:PaAprop1}) holds. Finally, if $f$ is differentiable at
$t_0$, it is easily seen that $d_n^{-2}(K(t_0+d_ns) - K(t_0) - d_n s
J'_b(t_0)\}$ converges to $\frac12 A s^2$ with $A = f'(t_0)/f(t_0)$.

\section*{Acknowledgments}
We kindly thank Sir David Cox for suggesting the problem of estimating
a spectral density under monotonicity assumptions. We would also like
to thank the Associate Editor and referees for their valuable
comments.

%suskaldyti doi

%
\printaddresses

\end{document}